\documentclass{article}
\usepackage[english]{babel}
\usepackage[utf8]{inputenc}
\usepackage{amsmath}
\usepackage{graphicx}
\usepackage{amsfonts}
\usepackage{enumitem}
\usepackage{amssymb}
\usepackage[colorinlistoftodos]{todonotes}
\usepackage{geometry}
\usepackage{amsthm}
\usepackage{hyperref}

\newcommand{\set}[1]{\left\{ #1 \right\}}
\newcommand{\setcond}[2]{\left\{ #1 \;\middle\vert\; #2 \right\}}

\newcommand{\RR}{\mathbb{R}}

\newcommand{\ZZ}{\mathbb{Z}}

\newcommand{\TT}{\mathbb{T}}

\newtheorem{thm}{Theorem}[section]
\newtheorem{lem}[thm]{Lemma}

\newtheorem*{rmk}{Remark}

\theoremstyle{definition}

\DeclareMathOperator{\ex}{ex}

\title{Sums of dilates over groups of prime order}
\author{
David Conlon\thanks{Department of Mathematics, Caltech, Pasadena, CA 91125, USA. Email: {\tt dconlon@caltech.edu}. Research supported by NSF Awards DMS-2054452 and DMS-2348859.} \and Jeck Lim\thanks{Department of Mathematics, Caltech, Pasadena, CA 91125, USA. Email: {\tt jlim@caltech.edu}. Research partially supported by an NUS Overseas Graduate Scholarship.}}
\date{}

\begin{document}
\maketitle

\begin{abstract}
For $p$ prime, $A \subseteq \mathbb{Z}/p\mathbb{Z}$ and $\lambda \in \mathbb{Z}$, the sum of dilates $A + \lambda \cdot A$ is defined by
\[A + \lambda \cdot A = \{a + \lambda a' : a, a' \in A\}.\]
The basic problem on such sums of dilates asks for the minimum size of $|A + \lambda \cdot A|$ for given $\lambda$, $A$ of given density $\alpha$, and $p$ tending to infinity. We investigate this problem for $\alpha$ fixed and $\lambda$ tending to infinity, proving near-optimal bounds in this case.
\end{abstract}

\section{Introduction}

Given subsets $A$ and $B$ of an abelian group $G$, their sumset $A+B$ is given by
\[A+B = \{a + b: a\in A, b \in B\}.\]
The difference set $A - B$ is defined similarly with subtraction replacing addition.

If $G = \mathbb{Z}$, then it is a simple exercise to show that $|A+B| \ge |A| + |B| - 1$. Indeed, if we order the elements of $A$ as $a_1 < a_2 < \dots < a_s$ and $B$ as $b_1 < b_2 < \dots < b_t$, then $A + B$ contains the elements
\[a_1 + b_1 < a_1 + b_2 < \dots < a_1 + b_t < a_2 + b_t < \dots < a_s + b_t,\]
so we have at least $|A| + |B| - 1$ distinct elements. If we work instead over $G = \ZZ/p\ZZ$ with $p$ prime, the corresponding inequality, known as the Cauchy--Davenport theorem~\cite{C13, D35}, says that
\[|A + B| \ge \min\{|A| + |B| - 1, p\},\]
since one must account for the possibility that the sumset contains all the elements of $\ZZ/p\ZZ$. Several proofs of this inequality are known (see, for example,~\cite{AMR95}), but, unlike the integer case, none of them is particularly simple.

Our concern in this paper will be with sumsets of a particular type. 
Given a subset $A$ of an abelian group $G$ and $\lambda \in \ZZ$, let
\[A + \lambda \cdot A = \{a + \lambda a' : a, a' \in A\}.\]
Such sums of dilates, as they are known, have attracted considerable attention in recent years, with the basic problem asking for an estimate on the minimum size of $|A + \lambda \cdot A|$ given $|A|$. Over the integers, this problem was essentially solved by Bukh~\cite{B08}, who showed that, for any finite set of integers $A$,
\[|A + \lambda \cdot A| \geq (|\lambda| + 1)|A| - o(|A|).\]
This result was later tightened by Balogh and Shakan~\cite{BS14}, improving the $o(|A|)$ term to a constant $C_\lambda$ depending only on $\lambda$, which is best possible up to the value of $C_\lambda$ (see also~\cite{CHS09, CSV10, DCS14, HR11, L13} for some earlier work on specific cases and~\cite{BS15, CL24, H21, KP23, S16} for extensions and variations). 

The analogous problem over $\ZZ/p\ZZ$ with $p$ prime was first studied in detail by Plagne~\cite{P11} and by Fiz Pontiveros~\cite{F13}. For instance, using a rectification argument, which allows one to treat small subsets of $\ZZ/p\ZZ$ as though they are sets of integers, the latter showed that for every $\lambda \in \ZZ$ there exists $\alpha > 0$ such that
\[|A + \lambda \cdot A| \geq (|\lambda| + 1)|A| - C_\lambda\]
for all $|A| \leq \alpha p$. On the other hand, he showed that for every $\lambda \in \ZZ$ and $\epsilon > 0$ there exists $\delta > 0$ such that, for every sufficiently large prime $p$, there is a set $A \subseteq \ZZ/p\ZZ$ with $|A| \geq (\frac12 - \epsilon)p$ such that $|A + \lambda \cdot A| \leq (1 - \delta)p$. That is, as $|A|$ approaches $p/2$, one cannot do much better than the Cauchy--Davenport theorem, which tells us that $|A + \lambda \cdot A| \geq 2|A|-1$.

For our purposes, it will be convenient to introduce some terminology. For $p$ prime, $\lambda\in\ZZ$ and $\alpha\in (0,1)$, we let 
$$\ex(\ZZ/p\ZZ,\lambda,\alpha)=\min\left\{|A+\lambda\cdot A|/p : A\subseteq\ZZ/p\ZZ, \, |A|\geq \alpha p\right\}$$ 
and then define $\ex(\lambda,\alpha)=\limsup_p \ex(\ZZ/p\ZZ,\lambda,\alpha)$. The problem of asymptotically estimating the minimum size of sums of dilates over $\ZZ/p\ZZ$ may then be rephrased as the problem of determining $\ex(\lambda,\alpha)$. This seems very difficult in full generality, though the results of Fiz Pontiveros described above imply that
\begin{itemize}
\item
 $\ex(\lambda,\alpha) = (|\lambda| + 1)\alpha$ for $\lambda$ fixed and $\alpha$ sufficiently small in terms of $\lambda$ and 
 \item
 $\ex(\lambda, \alpha) < 1$ for $\alpha < \frac12$.
\end{itemize}

Here we look at the case where $\alpha$ is fixed and $\lambda$ is allowed to grow. 
In rough terms, we wish to understand how small the sum of dilates $A + \lambda \cdot A$ can be if we fix the density $\alpha$ of $A$ and let $\lambda$ tend to infinity. 
More precisely, we set $\ex(\alpha)=\limsup_{\lambda\to\infty} \ex(\lambda, \alpha)$ and investigate the behavior of $\ex(\alpha)$. 

By Cauchy--Davenport, if $\alpha \ge \frac12$, then 
$\ex(\alpha)=1$. Moreover, if $\alpha \leq \frac12$, then, again by Cauchy--Davenport, $|A+\lambda\cdot A|\geq 2|A| - 1$, so $\ex(\alpha)\geq 2\alpha$. On the other hand, since $|A+\lambda\cdot A|\leq p$, we always have the trivial upper bound $\ex(\alpha)\leq 1$. Our main result improves these simple bounds significantly, giving a reasonably complete picture of the behavior of $\ex(\alpha)$. 

\begin{thm} \label{thm:main}
There exist constants $C, C', c > 0$ such that
\[e^{C' \log^c(1/\alpha)}\alpha \leq \ex(\alpha) \leq e^{C \sqrt{\log (1/\alpha)}}\alpha\] 
for all $\alpha \in (0, \frac12)$. Moreover, $\ex(\alpha)<1$ for all $\alpha \in (0, \frac12)$.  
\end{thm}

Unlike in the fixed $\lambda$ case, we cannot improve the trivial upper bound $\ex(\alpha)\leq 1$ by just taking $A$ to be an interval. Instead, what we do is show that $\ex(\alpha)$ is bounded above by a continuous variant defined over the torus $\TT=\RR/\ZZ$ and then provide an upper bound for that variant. We go straight into the details of this construction, before returning to the lower bound, which makes use of several classical tools from additive combinatorics, in Section~\ref{sec:lower}.

\section{The upper bound} \label{sec:upper} 

Let $\TT=\RR/\ZZ$, $n>1$ be an integer and $\mu$ be the Lebesgue measure on $\TT^k$ for any positive integer $k$. 
Let $\pi_1:\TT^n\to \TT^{n-1}$ be the projection map ignoring the first coordinate and $\pi_n:\TT^n\to \TT^{n-1}$ the projection map ignoring the last coordinate. Consider the following problem: given $0<\alpha<1$, what is the smallest possible value of $\mu(\pi_1(B)+\pi_n(B))$ over all open sets $B\subseteq \TT^n$ with $\mu(B) > \alpha$? 

Equivalently, we can ask for the smallest possible value of $\mu(B\times \TT+\TT\times B)$ over all open sets $B\subseteq \TT^n$ with $\mu(B) > \alpha$. In this form, written as a problem about sums of shifts rather than sums of projections, there is a ready analogy with the problem of estimating sums of transcendental dilates, which can also be phrased in terms of sums of shifts and ultimately has bounds of a similar form~\cite{CL23}. This analogy partly motivates the methods we use here for both the upper and lower bounds. 

To capture this question more succinctly, we define 
 \[\ex_T(n,\alpha)=\inf \left\{\mu(\pi_1(B)+\pi_n(B)) : B\subseteq \TT^n \text{ open}, \, \mu(B)>\alpha\right\}\] 
 and set $\ex_T(\alpha)=\lim_{n\to\infty} \ex_T(n,\alpha)$. This limit exists since $\ex_T(n,\alpha)$ is decreasing in $n$. Indeed, if $B\subseteq \TT^n$ with $\mu(\pi_1(B)+\pi_n(B))=\beta$, consider $B'=B\times \TT\subseteq\TT^{n+1}$. Then $\mu(B')=\mu(B)$ and $\mu(\pi_1(B')+\pi_{n+1}(B'))=\mu(\pi_1(B)\times \TT+B)=\mu(\pi_1(B)+\pi_n(B))=\beta$, so that $\ex_T(n+1,\alpha)\leq \ex_T(n,\alpha)$. 
 
 The main result of this section says that $\ex(\alpha) \le \ex_T(\alpha)$, thereby allowing us to give an upper bound on $\ex(\alpha)$ by instead bounding $\ex_T(\alpha)$. The idea of the proof is to construct an example in $\ZZ/p\ZZ$ from one in $\TT^n$ by approximating each point of $\TT^n$ by a number in $\ZZ/p\ZZ$ written in base $\lambda$, with each point $(x_1,\ldots,x_n)\in \TT^n$ roughly corresponding to $\lfloor (x_1+\frac{x_2}{\lambda}+\cdots +\frac{x_n}{\lambda^{n-1}})p\rfloor \in \ZZ/p\ZZ$. 

\begin{thm} \label{thm:semiequiv}
    $\ex(\alpha)\leq \ex_T(\alpha)$.
\end{thm}

\begin{proof}
    Let $n>1$ and $B\subseteq\TT^n$ be an open set such that $\mu(B)=\alpha'>\alpha$ 
    and $\mu(\pi_1(B)+\pi_n(B))=\beta$. We will show that $\ex(\alpha)\leq \beta$.

    Let $\epsilon>0$ be arbitrary, $\lambda$ be a positive integer, $T=\ZZ/\lambda \ZZ$ and discretize $\TT^n$ into $T^n$. For $x=(x_1,\ldots,x_n)\in T^n$ (with integers $0\leq x_i<\lambda$ for each $i$), define $C_x\subseteq \TT^n$ to be the cubical box
    \[\prod_{i=1}^n \left[ \frac{x_i}{\lambda},\frac{x_i+1}{\lambda}\right).\]
    Let $S=\setcond{x\in T^n}{C_x\subseteq B}$ and $B'=\bigcup_{x\in S} C_x\subseteq B$. As $\lambda\to\infty$, $\mu(B')$ approaches $\mu(B)=\alpha'$ since $B$ is open. Therefore, for $\lambda$ sufficiently large in terms of $\epsilon$, we have $\mu(B')\geq \alpha'-\epsilon$. For $x\in T^n$, define $I_x$ to be the interval $[y,y+\lambda^{-n})$, where 
    $$y=\frac{x_1}{\lambda}+\frac{x_2}{\lambda^2}+\dots+\frac{x_n}{\lambda^n}.$$ 
    Set $A=\bigcup_{x\in S} I_x\subseteq \TT$. Then $\mu(A)=|S|/\lambda^n=\mu(B')\geq \alpha'-\epsilon$. We claim that 
    $$\mu(A+\lambda\cdot A)\leq \mu(\pi_1(B')+\pi_n(B')).$$ 
    
    To see how the theorem follows from this claim, we again discretize $\TT$ into $\ZZ/p\ZZ$. Set $A'\subseteq \ZZ/p\ZZ$ to be $A'=\{0\leq a<p : \left[\frac{a}{p}, \frac{a+1}{p}\right)\subseteq A\}$. By construction, $|A'|/p\leq \mu(A)$. Moreover, since $A$ is a finite union of half-closed intervals, $|A'|/p$ approaches $\mu(A)$ as $p\to\infty$. Therefore, for $p$ sufficiently large in terms of $\epsilon$, we have $|A'|\geq (\mu(A)-\epsilon)p$. For any $a+\lambda b\in A'+\lambda\cdot A'$ with $a,b\in A'$, we have $\left[\frac{a}{p},\frac{a+1}{p}\right)\subseteq A$ and $\frac{b}{p}\in A$. Thus, $\left[\frac{a+\lambda b}{p},\frac{a+\lambda b+1}{p}\right)\subseteq A+\lambda\cdot A$. Hence, $|A'+\lambda\cdot A'|/p\leq \mu(A+\lambda\cdot A)$. From the claim,
    $$\frac{|A'+\lambda \cdot A'|}{p}\leq \mu(A+\lambda\cdot A)\leq \mu(\pi_1(B')+\pi_n(B'))\leq \beta.$$
    Since $|A'|\geq (\mu(A)-\epsilon)p\geq (\alpha'-2\epsilon)p$, taking $\epsilon=\frac{\alpha'-\alpha}{2}$ gives $\ex(\alpha) \leq \beta$, as required.

    To prove the claim, let $S'=\pi_1(S)+\pi_n(S)+\set{0,1}^{n-1}$. Then $S'$ is the set of all $z = (z_1,z_2,\dots,z_{n-1}) \in T^{n-1}$ with $z_k=a_{k+1}+b_k+\epsilon_k$ for some $a = (a_1, a_2, \dots, a_{n}), b = (b_1, b_2, \dots, b_{n})\in S$ and $\epsilon_k\in \set{0,1}$ for all $k$. Since $B'$ is the union of boxes $\bigcup_{x\in S} C_x$, we have that $\pi_1(B')+\pi_n(B')$ is the union of boxes $\bigcup_{x\in S'} C_x$, though now with each $C_x\subseteq \TT^{n-1}$.  
    Thus, 
    $$|S'|/\lambda^{n-1}=\mu(\pi_1(B')+\pi_n(B')).$$
    On the other hand, $A+\lambda\cdot A$ consists of all points in $\TT$ of the form 
    $$\frac{b_1+a_2}{\lambda}+\frac{b_2+a_3}{\lambda^2}+\dots +\frac{b_{n-1}+a_n}{\lambda^{n-1}}+\frac{b_n}{\lambda^n}+\epsilon,$$ 
    where $a,b\in S$ and $\epsilon\in [0,\lambda^{-n}+\lambda^{-n+1})$. Here, we are viewing $a_i$ and $b_i$ as integers in $[0,\lambda-1]$, so $b_i+a_{i+1}$ could ``overflow''. Nevertheless, each element of $A+\lambda\cdot A$ is of the form 
    $$\frac{c_1}{\lambda}+\frac{c_2}{\lambda^2}+\cdots+\frac{c_{n-1}}{\lambda^{n-1}}+\delta,$$ 
    where $c_i=b_i+a_{i+1}\bmod{\lambda}$ or $b_i+a_{i+1}+ 1\bmod{\lambda}$ and $\delta\in [0,\lambda^{-n+1})$. Thus, $A+\lambda\cdot A\subseteq \bigcup_{x\in S'}I_x$, so we have 
    $$\mu(A+\lambda\cdot A)\leq |S'|/\lambda^{n-1}=\mu(\pi_1(B')+\pi_n(B')),$$ 
    as required.
\end{proof}

We believe that the two functions $\ex(\alpha)$ and $\ex_T(\alpha)$ should in fact be equal, but leave the task of proving that 
$\ex(\alpha) \ge \ex_T(\alpha)$ as an open problem.

We now give an upper bound for $\ex_T(\alpha)$, and therefore  $\ex(\alpha)$, by considering a suitable set $B\subseteq\TT^n$.

\begin{thm} \label{thm:uppersmall}
    For any positive integer $d$, $\ex_T(\alpha)\leq 2^{d-1}\alpha^{1-1/d}$ for all $\alpha \in (0,2^{-d})$. In particular, there is a constant $C> 0$ such that $\ex_T(\alpha) \leq e^{C \sqrt{\log (1/\alpha)}}\alpha$ for all $\alpha \in (0, \frac12)$.
\end{thm}

\begin{proof}
    If $B=(0,\gamma^{1/d})^d\subseteq \TT^d$, then $\mu(B)=\gamma$. Furthermore, $\pi_1(B)=\pi_d(B)=(0,\gamma^{1/d})^{d-1}\subseteq \TT^{d-1}$, so we have $\mu(\pi_1(B)+\pi_d(B))=(2\gamma^{1/d})^{d-1}=2^{d-1}\gamma^{1-1/d}$. Taking the infimum over all $\gamma > \alpha$ then gives the required upper bound $\ex_T(\alpha)\leq \ex_T(d, \alpha)\leq 2^{d-1}\alpha^{1-1/d}$. 
    To get a general bound independent of $d$, we simply optimize by setting $d=\sqrt{\log (1/\alpha)}$ and the bound follows.
\end{proof}

\begin{rmk}
    The constant term $2^{d-1}$ in Theorem~\ref{thm:uppersmall} is not optimal. For example, for $d=3$, instead of picking $B$ to be the $\gamma^{1/3}\times \gamma^{1/3}\times \gamma^{1/3}$ box, we could optimize the side lengths of the box by picking $B$ to be $(2\gamma)^{1/3}\times (\gamma/4)^{1/3}\times (2\gamma)^{1/3}$. This yields $\mu(\pi_1(B)+\pi_3(B))=\frac{9}{2^{4/3}}\gamma^{2/3}$, where we note that $\frac{9}{2^{4/3}}<2^2$. We made no attempt to optimize these constants for higher values of $d$, as any improvement would not change the form of the bound $e^{C \sqrt{\log (1/\alpha)}}\alpha$.
\end{rmk}

While Theorem~\ref{thm:uppersmall} proves the first upper bound in Theorem~\ref{thm:main}, the following result proves the second upper bound $\ex(\alpha)<1$.

\begin{thm} \label{thm:upperlarge}
    $\ex_T(\alpha)<1$ for all $\alpha \in (0, \frac12)$.
\end{thm}

\begin{proof}
    Let $n$ be sufficiently large and set $B=\{x\in \TT^n : x_i>0 \mbox{ for all } i, \; \sum_{i=1}^n x_i < \frac{n}{2}-1\}$, where $x_i$ is considered an element of $[0,1)$ for all $i$. As $n\to\infty$, $\mu(B)\to \frac12$, since, if $x\in\TT^n$ is picked uniformly randomly, $\sum x_i$ is approximately normal with mean $\frac{n}{2}$ and variance $\Theta(n)$. Thus, for sufficiently large $n$, $\alpha<\mu(B)<\frac12$. Fix such an $n$. 
    Now both $\pi_1(B)$ and $\pi_n(B)$ are contained in the set $C=\{x\in\TT^{n-1} : \sum_{i=1}^{n-1} x_i < \frac{n}{2}-1\}$, so 
    $$\pi_1(B)+\pi_n(B)\subseteq C+C=\{x\in\TT^{n-1}:\sum_{i=1}^{n-1} x_i < n-2\}\subsetneq \TT^{n-1}.$$ 
    Hence, $\mu(\pi_1(B)+\pi_n(B))<1$, so that $\ex_T(\alpha) \le \ex_T(n, \alpha) <1$.
\end{proof}

\section{The lower bound} \label{sec:lower}

We now prove the lower bound in Theorem~\ref{thm:main}, which we restate as follows. As prefaced in the previous section, the proof of this result makes use of ideas similar to those used in~\cite{S08} for studying sums of transcendental dilates.

\begin{thm} \label{thm:lower}
There are constants $C',c >0$ such that $\ex(\alpha)\geq e^{C'\log^c (1/\alpha)}\alpha$ for all $\alpha\in (0,1/2)$. In particular, one may take $c= \frac 17$.
\end{thm}

In what follows, as well as the notation $\lambda \cdot B = \{\lambda b : b \in B\}$ for dilates, we will use $mB$ to denote the $m$-fold sumset
$$mB = \underbrace{B + B + \dots + B}_\text{$m$ times}.$$ 
Before proving Theorem~\ref{thm:lower}, we require the following result, a variant of the Pl\"unnecke--Ruzsa inequality~\cite{R89}, which states that if $A, B$ are finite subsets of an abelian group with $|A + B| \leq K |A|$, then $|mB - nB| \leq K^{m+n}|A|$.  
In particular, if $|B+B| \leq K |B|$, then $|mB| \leq K^m |B|$. 
Our result is a version of this latter inequality allowing for dilates of each term.

\begin{lem} \label{lem:pr}
Let $B$ be a finite subset of an abelian group, $\lambda$ an integer and $K>0$ such that $|B+\lambda\cdot B|\leq K|B|$. Then, for any positive integer $l$, 
\[|B+\lambda\cdot B+\lambda^2\cdot B+\cdots +\lambda^l\cdot B|\leq K^{7l-6}|B|.\]
\end{lem}

\begin{proof}
    The sum version of Ruzsa's triangle inequality~\cite{R96} states that, for any finite subsets $X,Y,Z$ of an abelian group,
    \[|X||Y+Z|\leq |X+Y||X+Z|.\]
    Taking $X=\lambda\cdot B$, $Y=Z=B$ and noting that $|\lambda\cdot B|=|B|$, we have $|B+B|\leq K^2|B|$. Hence, by the Pl\"unnecke--Ruzsa inequality, 
    $|B+B+B|\leq K^6|B|$. Thus, another application of Ruzsa's triangle inequality (with $X=B$, $Y=B+B$, $Z=\lambda\cdot B$) yields
    \[|B+B+\lambda\cdot B|\leq |B+B+B||B+\lambda\cdot B|/|B|\leq K^7|B|.\]

    We prove the lemma by induction on $l$, noting that the case $l=1$ follows from the given assumption. Suppose now that we have
    \[|B+\lambda\cdot B+\lambda^2\cdot B+\cdots +\lambda^l\cdot B|\leq K^{7l-6}|B|\]
    for some $l$ and we wish to prove it for $l+1$. Yet another application of Ruzsa's triangle inequality (with $X=\lambda^l\cdot B$, $Y=B+\lambda\cdot B+\cdots +\lambda^{l-1}\cdot B$, $Z=\lambda^l\cdot B+\lambda^{l+1}\cdot B$) yields
    \begin{align*}
        |B+\lambda\cdot B+\cdots +\lambda^{l+1}\cdot B| &\leq |B+\lambda\cdot B+\cdots +\lambda^l\cdot B||\lambda^l\cdot B+\lambda^l\cdot B+\lambda^{l+1}\cdot B|/|B|\\
        &\leq K^{7l-6}|\lambda^l\cdot B+\lambda^l\cdot B+\lambda^{l+1}\cdot B|\\
        &= K^{7l-6}|B+B+\lambda\cdot B|\\
        &\leq K^{7(l+1)-6}|B|,
    \end{align*}
    as required.
\end{proof}

The other thing that we need for the proof of Theorem~\ref{thm:lower} is Sanders' quantitative version of the Bogolyubov--Ruzsa lemma~\cite[Theorem 1.1]{S12}, which states that if $A$ is a finite subset of an abelian group with $|A+A|\le K|A|$, then $2A - 2A$ contains a proper generalized arithmetic progression $P$ of dimension $d\leq d_0(K) \leq C \log^6 K$ and size at least $C_1(K)|A|$, where $C$ is an absolute constant. Here a generalized arithmetic progression $P$ of dimension $d$ is a set of the form 
$$P = \{a + \sum_{i=1}^d n_i v_i : 0 \le n_i \le k_i - 1 \mbox{ for all } i\}$$
and such a generalized arithmetic progression is proper if all of its elements are distinct, that is, if $|P| = k_1 k_2 \cdots k_d$. 

\begin{proof}[Proof of Theorem~\ref{thm:lower}] 
    Fix $\alpha\in (0,1/2)$ and let $K=2\ex(\alpha)/\alpha$. 
    Let $\lambda$ be sufficiently large and $p$ be sufficiently large in terms of $\lambda$. Let $A\subseteq \ZZ/p\ZZ$, which we may assume has size $|A|=\alpha p$, 
    be such that $|A+\lambda\cdot A|\leq 2\ex(\alpha) p = K|A|$. By Ruzsa's triangle inequality, we again have $|A+A|\leq K^2|A|$. Hence, by Sanders' quantitative version of the Bogolyubov--Ruzsa lemma, $2A-2A$ contains a proper generalized arithmetic progression $P$ of dimension $d\leq d_0(K) \leq C \log^6 K$ and size at least $C_1(K)\alpha p$, where $C$ is an absolute constant. By the Pl\"unnecke--Ruzsa inequality, we have $|2A-2A+2A-2A|\leq K^{16}|A|$. By Ruzsa's triangle inequality (with $X=A$, $Y=2A-2A+A-2A$, $Z=\lambda\cdot A$), we have
    $$|2A-2A+A-2A+\lambda\cdot A|\leq |2A-2A+2A-2A||A+\lambda\cdot A|/|A|\leq K^{17}|A|.$$
    Repeating three more times, each time replacing an appropriate $A$ term with $\lambda\cdot A$, we get
    $$|(2A-2A)+\lambda\cdot (2A-2A)|\leq K^{20}|A|.$$
    By Lemma~\ref{lem:pr} applied to $2A-2A$, we then have that,
    \[|(2A-2A)+\lambda\cdot (2A-2A)+\cdots +\lambda^d\cdot (2A-2A)|\leq K^{140d}|A|.\]

    Suppose $P=v_0+P_0$ for some $v_0\in \ZZ/p\ZZ$ and $P_0$ a proper Minkowski sum of $d$ arithmetic progressions $\set{0,v_i,2v_i,\ldots,(k_i-1)v_i}$, $i=1,\ldots,d$, with $|P|=|P_0|=k_1k_2\cdots k_d$ and $k_1\geq k_2\geq\cdots \geq k_d$. Let $m\leq d$ be the largest integer with $k_m\geq \lambda$. Since $|P_0|\geq \lambda^d$ for sufficiently large $p$, we have $m\geq 1$. Let $P'=\sum_{i=1}^m \set{0,v_i,2v_i,\ldots,(k_i-1)v_i}$. Then this is a proper sum with $|P'|\geq |P_0|/\lambda^{d-m}$. Since $k_1\geq \cdots\geq k_m\geq \lambda$, we have that, 
    \begin{align*}
        P'+\lambda\cdot P'+\lambda^2\cdot P'+\dots+\lambda^d\cdot P' \supseteq \sum_{i=1}^m \set{0,v_i,2v_i,\ldots,\lambda^d (k_i-1)v_i} 
        =\lambda^d P'.
    \end{align*}
    By repeated application of the Cauchy--Davenport theorem, we have that 
    $$|\lambda^d P'|\geq \min(\lambda^d |P'|-\lambda^d+1,p)\geq \min(\lambda^{m}C_1 \alpha p-\lambda^d+1,p)=p$$ 
    for $\lambda$ large enough that $\lambda C_1\alpha\geq 2$ and $p$ sufficiently large. Thus, $P'+\lambda\cdot P'+\lambda^2\cdot P'+\dots+\lambda^d\cdot P' =\ZZ/p\ZZ$. On the other hand,
    \begin{align*}
        & |P'+\lambda\cdot P'+\lambda^2\cdot P'+\dots+\lambda^d\cdot P'|\\
        &\leq |P+\lambda\cdot P+\lambda^2\cdot P+\dots+\lambda^d\cdot P|\\
        &\leq |(2A-2A)+\lambda\cdot (2A-2A)+\lambda^2\cdot (2A-2A)+\dots +\lambda^d\cdot (2A-2A)|\\
        &\leq K^{140d}|A| \leq K^{140 d_0}|A|.
    \end{align*}
    This implies that $K^{140 d_0}\alpha\geq 1$. From $d_0\leq C\log^6 K$, we obtain $e^{140 C \log^{7} K}\alpha\geq 1$, which implies that 
    $$\ex(\alpha)= K\alpha/2 \geq e^{C'(\log\frac{1}{\alpha})^c}\alpha$$ 
    for some absolute constants $c$ and $C'$, where one may take $c = \frac 17$.
\end{proof}

If one could show that the Bogolyubov--Ruzsa lemma holds with $d_0(K) \leq C \log K$, which would be best possible, then we could take $c = \frac12$, matching our upper bound. 

\medskip

To close, let us mention a variant of the problem we have studied, namely, that of estimating the minimum size of $|A + \dots + A + \lambda \cdot A|$ over all $A \subseteq \ZZ/p\ZZ$ of given size. If there are $k$ summands, we can again study the asymptotic behaviour of this minimum by considering
\[\ex(k, \lambda, \alpha) = \limsup_{p \rightarrow \infty} \min\left\{| \underbrace{A + \dots + A}_{k-1\text{ times}} + \lambda \cdot A|/p : A\subseteq\ZZ/p\ZZ, \, |A|\geq \alpha p\right\}.\]
As a possible extension of his result that $\ex(\lambda, \alpha) < 1$ for $\alpha < \frac12$, Fiz Pontiveros~\cite[Conjecture 1.3]{F13} conjectured that $\ex(k, \lambda, \alpha) < 1$ for $\alpha < \frac1k$. However, this is easily seen to be false. 
Indeed, a simple consequence of the proof of Theorem~\ref{thm:lower} is that, provided $k$ is sufficiently large, $|A + \lambda \cdot A| \geq 10|A|$ for all $A \subseteq \ZZ/p\ZZ$ with $|A| = \lceil p/(k+1) \rceil$ and all $\lambda$ sufficiently large in terms of $k$. But then, by repeated application of the Cauchy--Davenport inequality, $|A + \dots + A + \lambda \cdot A| \geq \min\{(k+8)|A| - (k-2), p\} = p$. In particular, $\ex(k, \lambda, \alpha) = 1$ for $\alpha \geq 1/(k+1)$ and $\lambda$ sufficiently large in terms of $k$. This bound on the minimum $\alpha$ such that $\ex(k, \lambda, \alpha) = 1$ for $\lambda$ sufficiently large in terms of $k$ can certainly be improved, though we have made no serious attempt to do so here. Instead, we leave it as an open problem to give more precise estimates on how this threshold changes with $k$.

\end{document}